\documentclass{birkjour}
\newtheorem{Theorem}{Theorem}[section]
\newtheorem{Corollary}{Corollary}[section]
\newtheorem{Definition}{Definition}[section]
\newtheorem*{Definition*}{Definition}
\newtheorem{Example}{Example}[section]
\newtheorem{Lemma}{Lemma}[section]
\newtheorem{Proposition}{Proposition}[section]

\newtheorem{Remark}{Remark}[section]

\numberwithin{equation}{section}
\renewcommand{\proofname}{proof}

\usepackage{color}                    
\begin{document}
	
	%
	%
	%
	%
	%
	%
	%
	%
	%

	\title[A new geometric constant to compare $p$-angular and skew $p$-angular distances ]
	{A new geometric constant to compare \\ $p$-angular and skew $p$-angular distances  }
	
	\author[Yuxin Wang]{Yuxin Wang}

\address{%
School of Mathematics and physics\\
Anqing Normal University\\
Anqing 246133\\
China}

\email{y24060028@stu.aqnu.edu.cn}

\thanks{This work was completed with the support of Anhui Province Higher Education Science Research Project (Natural Science). 2023AH050487.}
\author{Qi Liu$^*$}
\address{School of Mathematics and physics\br
Anqing Normal University\br
Anqing 246133\br
China}
\address{ International Joint Research Center of Simulation and Control for
	Population Ecology of Yantze River in Anhui\br
	Anqing Normal University\br
	Anqing 246133\br
	China}
\email{liuq67@aqnu.edu.cn}

\author{Jinyu Xia}
\address{School of Mathematics and physics\br
	Anqing Normal University\br
	Anqing 246133\br
	China}
	\email{y23060036@aqnu.edu.cn}
	
	\author{Muhammad Sarfraz}
	\address{School of Mathematics and System Sciences\br
		 Xinjiang University\br
		 Urumgi 830046\br
		China}
	\email{sarfraz@xju.edu.cn}


	\date{---}
	\dedicatory{---}
	
	\begin{abstract}
		The $p$-angular distance was first introduced by Maligranda in 2006, while the skew $p$-angular distance was first introduced by Rooin in 2018. In this paper, we shall introduce a new geometric constant named Maligranda-Rooin constant in Banach spaces to compare $p$-angular distance and skew $p$-angular distance. We denote the Maligranda-Rooin constant as $\mathcal{MR}_p(\mathcal{X})$. First, the upper and lower bounds for the $\mathcal{MR}_p(\mathcal{X})$ constant is given. Next, it's shown that, a normed linear space is an inner space if and only if $\mathcal{MR}_p(\mathcal{X})=1$. Moreover,  an equivalent form of this new constant is established. By means of the $\mathcal{MR}_p(\mathcal{X})$ constant, we carry out the quantification of the characterization of uniform non-squareness. Finally, we study the relationship between the $\mathcal{MR}_p(\mathcal{X})$ constant, uniform convexity, uniform smooth and normal structure.
	\end{abstract}
	
	\maketitle
	\section{Introduction}
Throughout this article, we always regarded $\mathcal{X}$ be a real Banach space with $\operatorname{dim}$$\mathcal{X} \geq 2.$ The unit ball and
unit sphere of $\mathcal{X}$ are denoted by $\mathcal{B_X}$  and $S_X$, respectively.

Clarkson \cite{05}  introduced the concept of angular distance in a normed space in 1936:
$$\mathcal{\alpha}[x_1,x_2]=\left\|\frac{x_1}{\|x_1\|}-\frac{x_2}{\|x_2\|}\right\|, \quad \forall x_1, x_2 \in\{\mathcal{X},\|\cdot\|\}.$$
In 1964, Dunkl and Williams \cite{06} has shown that, $\forall x_1, x_2 \in\{\mathcal{X},\|\cdot\|\}$, the inequality
$$
\left\|\frac{x_1}{\|x_1\|} -\frac{x_2}{\|x_2\|} \right\| \leq \frac{4\|x_1-x_2\|}{\|x_1\|+\|x_2\|}
$$
holds for all  $x_1, x_2 \neq 0$. Moreover, they have proved that if $\mathcal{X}$ is a Hilbert space, then the above inequality is transformed into the following form:
$$
\left\|\frac{x_1}{\|x_1\|} -\frac{x_2}{\|x_2\|} \right\| \leq \frac{2\|x_1-x_2\|}{\|x_1\|+\|x_2\|}.
$$
Thus, in 2008, Jiménez-Melado \cite{07} et al defined the following constant to characterize Hilbert space:$$\mathcal{DW}(\mathcal{X})=\sup\bigg\{\frac{\|x_1\|+\|x_2\|}{\|x_1-x_2\|}\bigg\|\frac{x_1}{\|x_1\|}-\frac{x_2}{\|x_2\|}\bigg\|, x_1, x_2\in\mathcal{X}\backslash\{0\}, x_1\neq x_2\bigg\}.$$In the same paper, the authors introduced some properties of the $\mathcal{DW}(\mathcal{X})$ constant, we collect some of them:

(i) $2\leq\mathcal{DW}(\mathcal{X})\leq 4$, $\mathcal{DW}(\mathcal{X})=2$ if and only if $\mathcal{X}$ is a Hilbert space;

(ii) $\mathcal{X}$ is uniformly non-square if and only if $\mathcal{DW}(\mathcal{X})<4$;

(iii) If $\mathcal{DW}(\mathcal{X})<\sqrt[3]{3+2 \sqrt{2}}+\sqrt[3]{3-2 \sqrt{2}}$, then $X$ has normal structure.
\\For more details of the $\mathcal{DW}(\mathcal{X})$  constant, I suggest readers refer to \cite{07,08,09}.

On the basis of angular distance, in 2013, Dehghan \cite{10} introduced the concept of skew angular distance:$$\mathcal{\beta}[x_1,x_2]=\left\|\frac{x_2}{\|x_1\|}-\frac{x_1}{\|x_2\|}\right\|, \quad \forall x_1, x_2 \in\{\mathcal{X},\|\cdot\|\}.$$
In 2024, Fu et al \cite{11} relate the angular distance to the skew angular distance and establish the following constant:$$\mathcal{D R(X)}=\sup \left\{\frac{\alpha[x_1, x_2]}{\beta[x_1, x_2]}: x_1, x_2 \in \mathcal{X} \backslash\{0\}, \beta[x_1, x_2)] \neq 0\right\}.$$In the same paper, the authors introduced some properties of the $\mathcal{DR}(\mathcal{X})$ constant, we collect some of them:

(i) $1\leq\mathcal{DR}(\mathcal{X})\leq 2$, $\mathcal{DR}(\mathcal{X})=1$ if and only if $\mathcal{X}$ is a Hilbert space;

(ii) $\mathcal{X}$ is uniformly non-square if and only if $\mathcal{DR}(\mathcal{X})<2$;

(iii) If $\mathcal{DR}(\mathcal{X})< 1.134$, then $X$ has normal structure.
\\In 2016,  the authors \cite{02} generalized angular distance, and introduced $ p $-angular distance:
for any $p\in \mathcal{R}$, $$\mathcal{\alpha}_p \left[x_1,x_2\right]=\bigg\|\frac{x_1}{\|x_1\|^{1-p}}-\frac{x_2}{\|x_2\|^{1-p}}\bigg\|.$$ 
For more details of the $p$-angular distance, I suggest readers refer to \cite{12,13}.
\\Several years later, Jamal et al \cite{01} introduced the following skew $p$-angular distance: $$\mathcal{\beta}_p \left[x_1,x_2\right]=\bigg\|\frac{x_1}{\|x_2\|^{1-p}}-\frac{x_2}{\|x_1\|^{1-p}}\bigg\|.$$
Additionally, the authors \cite{01} has shown that a normed space $\mathcal{X}$ is
an inner product space, if and only if for any $p\leq1$ and $x_1,x_2\in \mathcal{X}\backslash\{0\}$, $\mathcal{\alpha}_{p}[x_1,x_2]\leq\mathcal{\beta}_{p}[x_1,x_2]$. 

In this paper, we are devoted to introduce a new geometric constant named Maligranda-Rooin constant to compare $p$-angular and skew $p$-angular distances. The arrangement of this article is as follows:

In the second part, we will review some basic concepts that are useful for the following parts.

In the third part, we gave the main definition of the $	\mathcal{MR}_p(\mathcal{X})$ constant, calculated its upper and lower bounds, as well as an equivalent form of it. In addition, we proved that a normed linear space is a Banach space if and only if $\mathcal{MR}_p(\mathcal{X})=1$.

In the last part, we consider  some relations with the Maligranda-Rooin constant and other geometric properties, including  uniform convexity, uniform smoothness and normal structure.
\section{Notations and Preliminaries}
\begin{Definition}\cite{17} A Banach space $\mathcal{X}$ is uniformly non-square if there exists  $\delta \in(0,1)$ such that for any pair $x_1, x_2 \in \mathcal{S_X}$, either $\|x_1+x_2\| \leq \delta$ or $\|x_1-x_2\| \leq \delta$.
\end{Definition}
\begin{Definition}\cite{18}
	The modulus of convexity $\delta_{\mathcal{X}}(\epsilon)$ of $\mathcal{X}$ is defined as
	$$
	\delta_{\mathcal{X}}(\epsilon)=\inf \left\{1-\left\|\frac{1}{2}(x_1+x_2)\right\|: x_1, x_2 \in \mathcal{B_X},\|x_1-x_2\| \geq \epsilon\right\}.
	$$
	The characteristic of
	convexity of $\mathcal{X}$ is defined as the number$$
	\epsilon_0(\mathcal{X})=\sup \left\{\epsilon \in[0,2]: \delta_\mathcal{X}(\epsilon)=0\right\} .
	$$
	
\end{Definition}
\begin{Definition}\cite{01}
	For $\mathcal{\tau}>0$, the module of smoothness $\mathcal{\rho}_\mathcal{X}(\mathcal{\tau})$ is defined as	
	$$
	\rho_\mathcal{X}(\mathcal{\tau})=\sup \left\{\frac{\|x_1+\mathcal{\tau}x_2\|+\|x_1-\mathcal{\tau} x_2\|}{2}-1: x_1, x_2 \in \mathcal{S}_\mathcal{X}\right\}.
	$$
	We say that $\mathcal{X}$ is uniformly smooth if $
	\rho'_\mathcal{X}(0)=\lim _{\tau \rightarrow 0^{+}} \frac{\rho_\mathcal{X}(\tau)}{\tau}=0
	$. We then say that $(\mathcal{X}, \|\cdot\|)$ is uniformly smooth.
\end{Definition}
\begin{Definition}\cite{14}
	We say that a Banach space $(\mathcal{X},\|\cdot\|)$ has normal structure \cite{14}, if for every nonempty, bounded, closed convex subset $\widetilde{\mathcal{X}}$ of $\mathcal{X}$,
	$$
	r(\widetilde{\mathcal{X}})=\inf _{x_2 \in \widetilde{\mathcal{X}}} \sup _{x_1 \in \widetilde{\mathcal{X}}}\|x_1-x_2\|<\operatorname{diam}(\widetilde{\mathcal{X}})=\sup _{x_1, x_2 \in \widetilde{\mathcal{X}}}\|x_1-x_2\| .
	$$	
\end{Definition}
\begin{Definition}\cite{15}
	The coefficient $\mu(\mathcal{X}) \in[1,3]$ was defined in by
	$$
	\mu(\mathcal{X})=\inf \left\{a>0: \lim \sup \left\|x_n+x_1\right\| \leq a \lim \sup \left\|x_n-x_1\right\|, x_n \stackrel{w}{\rightharpoonup} 0, x_1 \in \mathcal{X}\right\}.
	$$
\end{Definition}
\begin{Definition}\cite{16}
	The coefficient $\mathcal{M(X)}$ was introduced by T. Domínguez-Benavides  as
	$$
	\mathcal{M(X)}=\sup \left\{\frac{1+a}{\mathcal{R}(a, \mathcal{X})}: a \geq 0\right\},
	$$
	with
	$$
	\mathcal{R}(a, \mathcal{X})=\sup \left\{\liminf \left\|x_1+x_n\right\|\right\}
	$$where the supremum is taken over all $x_1 \in\mathcal{X}$ with $\|x_1\| \leq a$ and all weakly null sequences $\{x_n\}_{n=1}^{\infty}\in \mathcal{B}_\mathcal{X}$ such that
	$$
	\mathcal{D}\left(x_n\right)=\limsup _n\left(\limsup _m\left\|x_n-x_m\right\|\right) \leq 1.
	$$
	
\end{Definition}
\section{ The Maligranda-Rooin constant}
\begin{Definition}
	We define the following new constant: $$\begin{aligned}
		\mathcal{MR}_p(\mathcal{X})&=\sup\left\{\frac{\mathcal{\alpha}_p \left[x_1,x_2\right]}{\mathcal{\beta}_p \left[x_1,x_2\right]},~x_1,x_2\in \mathcal{X}\backslash\{0\},~\beta_p \left[x_1,x_2\right]\neq0\right\}\\&=\sup\left\{\frac{\bigg\|\frac{x_1}{\|x_1\|^{1-p}}-\frac{x_2}{\|x_2\|^{1-p}}\bigg\|}{\bigg\|\frac{x_1}{\|x_2\|^{1-p}}-\frac{x_2}{\|x_1\|^{1-p}}\bigg\|},~x_1,x_2\in \mathcal{X}\backslash\{0\},~\beta_p \left[x_1,x_2\right]\neq0\right\},
	\end{aligned}$$where $0\leq p\leq1$.
\end{Definition}
\begin{Remark}(i)If $p=0$, then $$	\mathcal{MR}_p(\mathcal{X})=\mathcal{DR}(\mathcal{X})=\sup\left\{\frac{\bigg\|\frac{x_1}{\|x_1\|}-\frac{x_2}{\|x_2\|}\bigg\|}{\bigg\|\frac{x_1}{\|x_2\|}-\frac{x_2}{\|x_1\|}\bigg\|},~x_1,x_2\in \mathcal{X}\backslash\{0\},~\beta \left[x_1,x_2\right]\neq0\right\}.$$
	(ii)If $p=1$, then for any Banach space $\mathcal{X}$, $\mathcal{MR}_p(\mathcal{X})=1$.
\end{Remark}
Now we will establish an equivalent form of the $\mathcal{MR}_p(\mathcal{X})$ constant.
\begin{Proposition}
	Let $\mathcal{X}$ be a Banach space. Then, $$\mathcal{MR}_p(\mathcal{X})=\sup\left\{\frac{\|\lambda x_1-\lambda^{p-1}x_2\|}{\|\lambda^{2-p}x_1-x_2\|},~x_1,x_2\in S_\mathcal{X},\lambda>0,\lambda \neq 1\right\},$$ where $0\leq p\leq1$.
\end{Proposition}
\begin{proof}On the one hand, since
	$$\begin{aligned}
		\frac{\left\|\frac{x_1}{\|x_1\| 1-p}-\frac{x_2}{\|x_2\| 1-p}\right\|}{\left\|\frac{x_1}{\|x_2\| 1-p}-\frac{x_2}{\|x_1\|^{1-p}}\right\|}&=\frac{\left\|\frac{x_1}{\|x_1\|} \cdot\| x_1\|^p-\frac{x_2}{\|x_2\|} \cdot\|x_2\|^p\right\|}{\left\|\frac{\|x_1\|}{\|x_2\|^{1-p}} \cdot \frac{x_1}{\|x_1\|}-\frac{\|x_2\|}{\|x_1\|^{-p}} \cdot \frac{x_2}{\|x_2\|}\right\|}\\&=\frac{\|x_1\|^p \bigg\|\frac{x_1}{\| x_1\|}-\frac{x_2}{\| x_2\|}\cdot\left(\frac{\|x_2\|}{\|x_1\|}\right)^p\bigg \|}{\frac{\|x_2\|}{\left\|x_1\right\|^{1-p}}\left\|\frac{x_1}{\|x_1\|} \cdot \frac{\|x_1\|}{\left\|x_2\right\|^{1-p}} \cdot \frac{\|x_1\|^{1-p}}{\left\|x_2\right\|}-\frac{x_2}{\| x_2}\right\|}\\&=\frac{\| x_1\|}{\| x_2\|} \cdot \frac{\left\|\frac{x_1}{\|x_1\|}-\left(\frac{\|x_2\|}{\| x_1\|}\right)^p \cdot \frac{x_2}{\|x_2\|}\right\|}{\left\|\left(\frac{\|x_1\|}{\| x_2\|}\right)^{2-p} \cdot \frac{x_1}{\| x_1\|}-\frac{x_2}{\|x_2\|}\right\|},
	\end{aligned}
	$$
	then let $\lambda=\frac{\|x_1\|}{\|x_2\|}$, we get that
	$$\begin{aligned}
		\frac{\left\|\frac{x_1}{\|x_1\| 1-p}-\frac{x_2}{\|x_2\| 1-p}\right\|}{\left\|\frac{x_1}{\|x_2\| 1-p}-\frac{x_2}{\|x_1\|^{1-p}}\right\|}&=\frac{\lambda\left\|x_1-\lambda^{-p} x_2\right\|}{\left\|\lambda^{2-p} x_1-x_2\right\|}=\frac{\left\|\lambda x_1-\lambda^{1-p} x_2\right\|}{\left\|\lambda^{2-p} x_1-x_2\right\|}\\&\leq\sup \left\{\frac{\| \lambda x_1-\lambda^{1-p} x_2\|}{\left\|\lambda^{2-p} x_1-x_2\right\|}:x_1, x_2 \in S_{\mathcal{X}}, \lambda>0, \lambda \neq 1\right\},
	\end{aligned}$$
	which means that $$\mathcal{MR}_p(\mathcal{X})\leq\sup\left\{\frac{\|\lambda x_1-\lambda^{p-1}x_2\|}{\|\lambda^{2-p}x_1-x_2\|},~x_1,x_2\in S_\mathcal{X},\lambda>0,\lambda\neq 1\right\}.$$
	On the other hand, for any $x_1,x_2\in S_{\mathcal{X}}$ and $\lambda$ with $\lambda>0,\lambda\neq 1$, we put $u=\lambda^{\frac1p}x_1$, $v=\lambda^{\frac{1-p}{p}}x_2$, thus, it's clear that $u,v\in \mathcal{X}\backslash 0$ with $\beta_{p}[u,v]\neq0$ and $$\frac{\left\|\lambda x_1-\lambda^{1-p} x_2\right\|}{\left\|\lambda^{2-p} x_1-x_2\right\|}=\frac{\bigg\| \frac{u}{\|u\|^{1-p}}-\frac{v}{\|v\|^{1-p}}\bigg \|}{\bigg\| \frac{u}{\|v\|^{1-p}}-\frac{v}{\|u\|^{1-p}}\bigg\|} \leq \mathcal{M R}_p(X).$$
	This implies that $$\mathcal{MR}_p(\mathcal{X})\geq\sup\left\{\frac{\|\lambda x_1-\lambda^{p-1}x_2\|}{\|\lambda^{2-p}x_1-x_2\|},~x_1,x_2\in S_\mathcal{X},\lambda>0,\lambda\neq 1\right\}.$$
	This completed the proof.
\end{proof}

Next, we are going to make an estimation of the value of this constant in any Banach space.
\begin{Proposition}\label{p1}
	Let $\mathcal{X}$ be a Banach space. Then, $$1\leq \mathcal{MR}_p(\mathcal{X})\leq 2.$$
\end{Proposition}
Before we can prove this proposition, we must first have recourse to the following crucial lemmas.
\begin{Lemma} \cite{01}
	(i) If $\frac{p}{2-p} \geq 1$, then
	\begin{equation}
		\mathcal{\alpha}_p[x_1, x_2] \\
		\leq \frac{p}{2-p} \max \left(\|x_1\|^{p-1}\|x_2\|^{1-p},\|x_2\|^{p-1}\|x_1\|^{1-p}\right) \beta_p[x_1, x_2].
	\end{equation}
	
	(ii) If $0\leq{\frac{p}{2-p}}\leq1,$ then \begin{equation}\label{e1}
		\mathcal{\alpha}_{p}[x_1,x_2]\leq\frac{4-3p}{2-p}\cdot\frac{\mathcal{\beta}_{p}[x_1,x_2]}{\max(\|x_1\|^{p-1}\|x_2\|^{1-p},\|x_2\|^{p-1}\|x_1\|^{1-p})}.
	\end{equation}
	
	(iii) If $\frac{p}{2-p}\leq0$, then\begin{equation}\label{e2}
		\mathcal{\alpha}_{p}[x_1,x_2]\leq\frac{4-3p}{2-p}\cdot\frac{\max(\|x_1\|^{p},\|x_2\|^{p})}{\max(\|x_1\|\|x_2\|^{p-1},\|x_2\|\|x_1\|^{p-1})}\mathcal{\beta}_{p}[x,y].
	\end{equation}
\end{Lemma}
\begin{Lemma}\label{l1}
	Let $0\leq p\leq1$ is a real number, and $\mathcal{X}$ be a Banach space. Then for any $x_1,x_2\in \mathcal{X}\backslash\{0\}$,  $$\frac{1}{\max(\|x_1\|^{p-1}\|x_2\|^{1-p},\|x_2\|^{p-1}\|x_1\|^{1-p})}\leq1.$$Additionally, $\mathcal{\alpha}_{p}[x_1,x_2]\leq\frac{4-3p}{2-p}\mathcal{\beta}_{p}[x_1,x_2]$ is always valid for any $0\leq p\leq1$.
\end{Lemma}
\begin{proof}
	In the following proof, we always let $0\leq p\leq1$,	to prove $$\frac{1}{\max(\|x_1\|^{p-1}\|x_2\|^{1-p},\|x_2\|^{p-1}\|x_1\|^{1-p})}\leq1,$$ we denote $\mathcal{A}=\max(\|x_1\|^{p-1}\|x_2\|^{1-p},\|x_2\|^{p-1}\|x_1\|^{1-p})$, then we consider the following three cases:
	
	\textbf{Case i:} Assume that $\|x_1\|=\|x_2\|$, then it's clear that $\mathcal{A}=1$, thus, $\frac{1}{\mathcal{A}}=1$.
	
	\textbf{Case ii:} Assume that $\|x_1\|>\|x_2\|$, then we have $\bigg(\frac{\|x_2\|}{\|x_1\|}\bigg)^{1-p}<\bigg(\frac{\|x_1\|}{\|x_2\|}\bigg)^{1-p}$, thus, $\mathcal{A}=\bigg(\frac{\|x_1\|}{\|x_2\|}\bigg)^{1-p}>1$, it follows that $\frac{1}{\mathcal{A}}< 1$.
	
	\textbf{Case iii:} Assume that$\|x_1\|<\|x_2\|$, then we have $\bigg(\frac{\|x_2\|}{\|x_1\|}\bigg)^{1-p}>\bigg(\frac{\|x_1\|}{\|x_2\|}\bigg)^{1-p}$, thus, $\mathcal{A}=\bigg(\frac{\|x_2\|}{\|x_1\|}\bigg)^{1-p}>1$, it follows that $\frac{1}{\mathcal{A}}< 1$.
	
	From the above, we obtain that $$\frac{1}{\max(\|x_1\|^{p-1}\|x_2\|^{1-p},\|x_2\|^{p-1}\|x_1\|^{1-p})}\leq1.$$
	Additionally, according to the inequality \eqref{e1}, it follows that $$\mathcal{\alpha}_{p}[x_1,x_2]\leq\frac{4-3p}{2-p}\mathcal{\beta}_{p}[x_1,x_2].$$
\end{proof}
\noindent	\textbf{Proof of Proposition \ref{p1}}
First, let $x_1=-x_2$, then we have $$\frac{\mathcal{\alpha}_p \left[x_1,x_2\right]}{\mathcal{\beta}_p \left[x_1,x_2\right]}=1,$$which implies that $\mathcal{MR}_p(\mathcal{X})\geq 1$.

Conversely, by the above Lemma \ref{l1}, we know that $\mathcal{\alpha}_{p}[x_1,x_2]\leq\frac{4-3p}{2-p}\mathcal{\beta}_{p}[x_1,x_2]$ is always valid for any $0\leq p\leq1$, and notice that, $\frac{4-3p}{2-p}$ is a decreasing function with respect to $p$, which means that $\frac{4-3p}{2-p}\leq 2$ holds for any  $0\leq p\leq1$. Thus, we get that $$\mathcal{MR}_p(\mathcal{X})\leq \frac{4-3p}{2-p}\leq2.$$

In the next place, we will give an example to show that the upper bound of the $\mathcal{MR}_p(\mathcal{X})$ constant is sharp.
\begin{Example}
	Let $\mathcal{X}_1=(\mathcal{R}^{n},\|\cdot\|_1)$, where the norm $\|\cdot\|_1$ is defined by $$\|x_i\|_1=\|(x_{i1},x_{i2})\|_1=|x_{i1}|+|x_{i2}|,$$ then $\mathcal{MR}_p(\mathcal{X})=2$.
\end{Example}
\begin{proof}
	Consider  $x_1=(0,1)$ and $x_2=(r,1)$, where $r$ is a sufficiently small positive real number. Then we have $$\begin{aligned}
		\bigg\|\frac{x_1}{\|x_1\|^{1-p}}-\frac{x_2}{\|x_2\|^{1-p}}\bigg\|&=\frac{r}{(1+r)^{1-p}}+1-\frac{1}{(1+r)^{1-p}}\\&=(1+r)^p+1,
	\end{aligned}$$
	and$$\begin{aligned}
		\bigg\|\frac{x_1}{\|x_2\|^{1-p}}-\frac{x_2}{\|x_1\|^{1-p}}\bigg\|&=\frac{1}{(1+r)^{1-p}}-r+1-\frac{1}{(1+r)^{1-p}}\\&=1-r.
	\end{aligned}$$
	Then, let $r\to 0$, we get that $$\frac{\bigg\|\frac{x_1}{\|x_1\|^{1-p}}-\frac{x_2}{\|x_2\|^{1-p}}\bigg\|}{	\bigg\|\frac{x_1}{\|x_2\|^{1-p}}-\frac{x_2}{\|x_1\|^{1-p}}\bigg\|}=2,$$which means that $\mathcal{MR}_p(\mathcal{X})\geq2$. Therefore, combined with the upper bound of the $\mathcal{MR}_p(\mathcal{X})$ constant in Proposition \ref{p1}. 
	
	We completed the proof.
\end{proof}

Next, based on the above proposition, we will show that	when the $\mathcal{MR}_p(\mathcal{X})$ constant satisfies certain conditions, it can ensure that a normed linear space must be an inner product space, as shown in the following proposition.
\begin{Theorem}\label{p2}
	Let $\mathcal{X}$ be a normed linear space. Then,	$\mathcal{X}$ is an inner product space if and only if  $\mathcal{MR}_p(\mathcal{X})=1$.
\end{Theorem}

\begin{proof}
	Since it has been proven that $\mathcal{X}$ is
	an inner product space if and only if   $\mathcal{\alpha}_{p}[x_1,x_2]\leq\mathcal{\beta}_{p}[x_1,x_2]$ for any $p\leq1$, it follows  that $\mathcal{MR}_p(\mathcal{X})\leq1$. Then combined with the lower bound of the $\mathcal{MR}_p(\mathcal{X})$ constant in Proposition \ref{p1}. 
	
	We completed the proof.
\end{proof}

\section{The relations with the Maligranda-Rooin constant and other geometric properties of Banach spaces} 	
In this section, we will show  some relations with the Maligranda-Rooin constant and other geometric properties, including  uniform convexity, uniform smoothness and normal structure. Before starting, we need the following crucial lemma:
\begin{Lemma}\label{l2}
	Let $\mathcal{X}$ be a Banach space, $\mathcal{X}^*$ is denoted as the dual space of $\mathcal{X}$.   If $\left\{f_{n}\right\}_{n=1}^{\infty}\subseteq \mathcal{B}_{\mathcal{X}^*}$ and $\left\{x_n\right\}_{n=1}^{\infty}\subseteq \mathcal{B}_{\mathcal{X}}$ such that $\lim _{n \rightarrow \infty} g_n\left(x_n\right)=1$, then, for any sequence $\left\{g_n\right\}_{n=1}^{\infty}\subseteq \mathcal{B}_{\mathcal{X}^*}$ with $\liminf _{n \rightarrow \infty} g_n\left(x_n\right)>0$, we have$$
	\mathcal{MR}_p(X) \geq \max \left\{\liminf _{n \rightarrow \infty}\left\|g_n\left(x_n\right) f_n-g_n\right\|, 1\right\} \geq\max \left\{\liminf _{n \rightarrow \infty} g_n\left(x_n\right)\left\|f_n-g_n\right\|, 1\right\} .
	$$
\end{Lemma}
\begin{proof}
	First, let us prove that $
	\mathcal{MR}_p(X) \geq \max \left\{\liminf _{n \rightarrow \infty}\left\|g_n\left(x_n\right) f_n-g_n\right\|, 1\right\}.$ We consider the following two cases:
	
	\textbf{Case i:} If $\liminf _{n \rightarrow \infty}\left\|g_n\left(x_n\right) f_n-g_n\right\|\leq 1$, then $$\max \left\{\liminf _{n \rightarrow \infty}\left\|g_n\left(x_n\right) f_n-g_n\right\|, 1\right\}=1,$$ and
	from Proposition \ref{p1}, we know that $	\mathcal{MR}_p(X)\geq1 $ is always valid. Thus, it follows that $
	\mathcal{MR}_p(X) \geq \max \left\{\liminf _{n \rightarrow \infty}\left\|g_n\left(x_n\right) f_n-g_n\right\|, 1\right\}.$
	
	\textbf{Case ii:} If $\liminf _{n \rightarrow \infty}\left\|g_n\left(x_n\right) f_n-g_n\right\|> 1$, then $$\max \left\{\liminf _{n \rightarrow \infty}\left\|g_n\left(x_n\right) f_n-g_n\right\|, 1\right\}=\liminf _{n \rightarrow \infty}\left\|g_n\left(x_n\right) f_n-g_n\right\|.$$
	
	Given $\eta \in\left(1, \liminf _{n \rightarrow \infty}\left\|g_n\left(x_n\right) f_n-g_n\right\|\right)$, then there exists positive integer $\mathcal{N^*}\geq 1$ such that, for all $n \geq \mathcal{N^*}$, the inequality $\left\|g_n\left(x_n\right) f_n-g_n\right\|>\eta$ holds and we can then find $y_n \in \mathcal{S}_\mathcal{X}$ such that $\left(g_n\left(x_n\right) f_n-g_n\right)\left(y_n\right)>\eta$. Moreover, since $\lim _{n \rightarrow \infty} f_n\left(x_n\right)=1$ and $\left\{x_n\right\}_{n=1}^{\infty} \subseteq \mathcal{B}_{\mathcal{X}}$, we obtain that $\lim _{n \rightarrow \infty}\left\|x_n\right\|=1$. Then, for $t \in(0,1)$, we have
	$$
	\lim _{n \rightarrow \infty}\left\|x_n+t y_n\right\| \geq \lim _{n \rightarrow \infty}\left|\left\|x_n\right\|-t\right|=1-t>0
	$$
	Thus, there exists $\mathcal{N}_1>\mathcal{N^*}$ such that for any $n \geq \mathcal{N}_1$, we define $$w_n=(x_n+ty_n)\|x_n+ty_n\|^{\frac{p-1}{2-p}}\neq 0.$$ In addition, since
	$$
	\begin{aligned}
		\lim _{n \rightarrow \infty}\left\|\frac{w_n}{\left\|x_n\right\|^{1-p}}-\frac{x_n}{\left\|w_n\right\|^{1-p}}\right\| & =	\lim _{n \rightarrow \infty}\frac{1}{\|w_n\|^{1-p}}\bigg\|\frac{w_n\|w_n\|^{1-p}}{\|x_n\|^{1-p}}-x_n\bigg\|
		\\&=\lim _{n \rightarrow \infty}\left\|x_n+t y_n\right\|^{\frac{p-1}{2-p}}\left\|\frac{x_n+t y_n}{\left\|x_n\right\|^{1-p}}-x_n\right\| \\
		& \geq (1+t)^{\frac{p-1}{2-p}} t>0.
	\end{aligned}
	$$
	Thus, there exists $\mathcal{N}_2>\mathcal{N}_1$ such that for any $n \geq \mathcal{N}_2$, $\left\|\frac{w_n}{\left\|x_n\right\|^{1-p}}-\frac{x_n}{\left\|w_n\right\|^{1-p}}\right\|>0$ is always holds, which means that
	$$
	\left\|\frac{w_n}{\left\|w_n\right\|^{1-p}}-\frac{x_n}{\left\|x_n\right\|^{1-p}}\right\|\left\|\frac{w_n}{\left\|x_n\right\|^{1-p}}-\frac{x_n}{\left\|w_n\right\|^{1-p}}\right\|^{-1}
	$$
	is well defined for $n$ large enough according to the above discussion.
	\\Then by the definition of $\mathcal{MR}_p(\mathcal{X})$, for each  $n \geq \mathcal{N}_2$, we obtain that $$
	\mathcal{M R}_p(\mathcal{X}) \geq \frac{\bigg\| \frac{w_n}{\left\|w_n\right\|^{1-p}}-\frac{x_n}{\left\|x_n\right\|^{1- p} }\bigg\|}{\left\|\frac{w_n}{\left\|x_n\right\|^{1-p}}-\frac{x_n}{\left\|w_n\right\|^{1- p}}\right\|}=\frac{\left\|w_n-\frac{\left\|w_n\right\|^{1-p} x_n}{\left\|x_n\right\|^{1-p}}\right\|}{\left\|\frac{w_n\left\|w_n\right\|^{1-p}}{\left\|x_n\right\|}-x_n\right\|}.
	$$
	Since $\left\{x_n\right\}_{n=1}^{\infty} \subseteq \mathcal{B}_{\mathcal{X}}$ and $\lim _{n \rightarrow \infty} f_n\left(x_n\right)=1$, it must be $\lim _{n \rightarrow \infty}\left\|x_n\right\|=1$ and therefore$$
	\begin{aligned}
		\mathcal{M R}_p(\mathcal{X}) & \geq \lim _{n \rightarrow \infty} \inf \frac{\left\|w_n-\frac{\left\|w_n\right\|^{1-p} x_n}{\left\|x_n\right\|^{1-p}}\right\|}{\left\|\frac{w_n\left\|w_n\right\|^{1-p}}{\left\|x_n\right\|}-x_n\right\|} \\
		& =\lim _{n \rightarrow \infty} \inf \frac{\left\|w_n-\right\| w_n\left\|^{1-p} x_n\right\|}{\left\|w_n\right\| w_n\left\|^{1-p}-x_n\right\|}\\&=\frac{\lim _{n \rightarrow \infty} \inf \bigg\|\left(x_n+t y_n\right)\| x_n+t y_n\|^{\frac{p-1}{2-p}}-\| x_n+t y_n\|^{\frac{1-p}{2-p}} x_n\bigg\|}{t}\\&=\frac{\lim _{n \rightarrow \infty} \inf\|x_n+t y_n\|^{\frac{p-1}{2-p}}\bigg\|x_n+t y_n-\| x_n+t y_n\|^{\frac{2(1-p)}{2 -p}} x_n\bigg\|}{t},
	\end{aligned}
	$$
	then since $0<p<1$, we know that $\|x_n+ty_n\|^{\frac{2(1-p)}{2-p}}\leq\|x_n+ty_n\|$ is always holds. Hence, we obtain that $$\begin{aligned}
		&\frac{\lim _{n \rightarrow \infty} \inf\|x_n+t y_n\|^{\frac{p-1}{2-p}}\bigg\|x_n+t y_n-\| x_n+t y_n\|^{\frac{2(1-p)}{2 -p}} x_n\bigg\|}{t}\\\geq&\frac{\lim _{n \rightarrow \infty} \inf\|x_n+t y_n\|^{\frac{p-1}{2-p}}\bigg\|x_n+t y_n-\| x_n+t y_n\| x_n\bigg\|}{t}\\\geq&\frac{\lim _{n \rightarrow \infty} \inf(1+t) ^{\frac{p-1}{2-p}}\bigg\|x_n+t y_n-\| x_n+t y_n\| x_n\bigg\|}{t}.
	\end{aligned}$$
	Moreover, for each $n \geq \mathcal{N}_2$, we have $$\begin{aligned}\frac{\bigg\|x_n+ty_n-\|x_n+ty_n\|x_n\bigg\|}{t}&=\left\|\left(\left\|x_n+t y_n\right\|-1\right) \frac{1}{t} x_n-y_n\right\|\\& 
		\geq \left(f_n\left(x_n\right)+t f_n\left(y_n\right)-1\right) \frac{1}{t} g_n\left(x_n\right)-g_n\left(y_n\right) \\& 
		= \frac{1}{t}\left(f_n\left(x_n\right)-1\right) g_n\left(x_n\right)+\left[g_n\left(x_n\right) f_n-g_n\right]\left(y_n\right) \\ &
		= \frac{1}{t}\left(f_n\left(x_n\right)-1\right) g_n\left(x_n\right)+\eta.
	\end{aligned}$$
	Consequently, for any $t \in(0,1)$ and $\eta \in\left(1, \liminf _{n \rightarrow \infty}\left\|g_n\left(x_n\right) f_n-g_n\right\|\right)$, we have
	$$
	\mathcal{MR}_p(\mathcal{X}) \geq \liminf _{n \rightarrow \infty}(1+t)^{\frac{p-1}{2-p}}\left[\frac{1}{t}\left(f_n\left(x_n\right)-1\right) g_n\left(x_n\right)+\eta\right]=(1+t)^{\frac{p-1}{2-p}} \eta.
	$$
	Let $t \rightarrow 0$ and $\eta \rightarrow \liminf _{n \rightarrow \infty}\left\|g_n\left(x_n\right) f_n-g_n\right\|$, we obtain
	$$
	\mathcal{M R}_p(\mathcal{X}) \geq \liminf _{n \rightarrow \infty}\left\|g_n\left(x_n\right) f_n-g_n\right\|=\max \left\{\liminf _{n \rightarrow \infty}\left\|g_n\left(x_n\right) f_n-g_n\right\|, 1\right\}.
	$$
	Combining the above two cases, we have completed the proof of  $$
	\mathcal{MR}_p(\mathcal{X}) \geq \max \left\{\liminf _{n \rightarrow \infty}\left\|g_n\left(x_n\right) f_n-g_n\right\|, 1\right\}.$$
	On the other hand, in \cite{03}, the author has shown that $$\max \left\{\liminf _{n \rightarrow \infty}\left\|g_n\left(x_n\right) f_n-g_n\right\|, 1\right\} \geq \liminf _{n \rightarrow \infty} g_n\left(x_n\right)\left\|f_n-g_n\right\|,$$
	hence, we completed the whole proof.
	
\end{proof}
\begin{Theorem}\label{t2}
	Let $\mathcal{X}$ be a Banach space. Then, $$\mathcal{MR}_p(\mathcal{X}) \geq \max\{\epsilon_0(\mathcal{X}),1\}.$$
\end{Theorem}
\begin{proof}
	Similarly, we consider two cases to prove this theorem.
	
	\textbf{Case i:} If $\epsilon_0(\mathcal{X})\leq 1$, then $\max\{\epsilon_0(\mathcal{X}),1\}=1,$ and
	from Proposition \ref{p1}, we know that $\mathcal{MR}_p(\mathcal{X})\geq1 $ is always valid. Thus, it follows that $$
	\mathcal{MR}_p(\mathcal{X}) \geq \max\{\epsilon_0(\mathcal{X}),1\}.$$
	
	\textbf{Case ii:} If $\epsilon_0(\mathcal{X})> 1$, then there exist two sequences $\{r_n\}_{n=1}^{\infty}$,$\{s_n\}_{n=1}^{\infty}\subseteq \mathcal{S}_\mathcal{X}$ such that $\|r_n-s_n\|\to\epsilon_0(\mathcal{X})$ and $\|r_n+s_n\|\to2.$ Thus, for each $n\geq 1$, according to Hahn-Banach theorem, there exist $ \widetilde{f}_n, \widetilde{g}_n \in \mathcal{S}_{\mathcal{X}^*}$ such that $\widetilde{f}_n\left(r_n+s_n\right)=\left\|r_n+s_n\right\|$ and $\widetilde{g}_n\left(r_n-s_n\right)=\left\|r_n-s_n\right\|$. Observe that
	$$
	\lim _{n \rightarrow \infty} \widetilde{f}_n\left(r_n\right)=\lim _{n \rightarrow \infty} \widetilde{f}_n\left(s_n\right)=1
	$$
	since
	$$
	\lim _{n \rightarrow \infty}\left(\widetilde{f}_n\left(r_n\right)+\widetilde{f}_n\left(s_n\right)\right)=\lim _{n \rightarrow \infty}\left\|r_n+s_n\right\|=2
	$$
	and $\left|\widetilde{f}_n\left(r_n\right)\right| \leq 1,\left|\widetilde{f}_n\left(s_n\right)\right| \leq 1$.
	In addition,
	$$
	\liminf _{n \rightarrow \infty} \widetilde{g}_n\left(r_n\right)=\liminf _{n \rightarrow \infty}\left(\left\|r_n-s_n\right\|+\widetilde{g}_n\left(s_n\right)\right) \geq\epsilon_0(\mathcal{X})-1>0
	$$
	and hence, by Lemma \ref{l2},
	$$
	\begin{aligned}
		\mathcal{MR}_p(\mathcal{X}) & \geq \liminf _{n \rightarrow \infty}\left\|\widetilde{g}_n\left(r_n\right) \widetilde{f}_n-\widetilde{g}_n\right\| \\
		& \geq  \liminf _{n \rightarrow \infty}\left(\widetilde{g}_n\left(r_n\right) \widetilde{f}_n\left(s_n\right)-\widetilde{g}_n\left(s_n\right)\right) \\
		& =\lim _{n \rightarrow \infty} \widetilde{g}_n\left(r_n-s_n\right) \\
		& = \lim _{n \rightarrow \infty}\left\|r_n-s_n\right\|= \epsilon_0(\mathcal{X}) .
	\end{aligned}
	$$
	This implies that  $$
	\mathcal{MR}_p(\mathcal{X}) \geq \max\{\epsilon_0(\mathcal{X}),1\}.$$
\end{proof}
\begin{Corollary}
	For any $0\leq p<1$, a Banach space $\mathcal{X}$ is uniformly non-square if and only if $$\mathcal{MR}_p(\mathcal{X})<2.$$
\end{Corollary}
\begin{proof}
	As shown above, a Banach space $\mathcal{X}$ is not uniformly non-square if and only if $\epsilon_0(\mathcal{X})=2,$ combine with Theorem \ref{t2}, we know that :
	
	A Banach space $\mathcal{X}$ is not uniformly non-square if and only if $\mathcal{MR}_p(\mathcal{X})=2,$ which means that $\mathcal{X}$ is  uniformly non-square if and only if $\mathcal{MR}_p(\mathcal{X})<2.$
\end{proof}
\begin{Theorem}\label{t3}
	Let $\mathcal{X}$ be a Banach space. Then, $$\mathcal{MR}_p(\mathcal{X}) \geq \max\bigg\{2\rho'_0(\mathcal{X}),1\bigg\}.$$
\end{Theorem}
\begin{proof}
	The inequality $\mathcal{MR}_p(\mathcal{X})\geq1$ always holds. We have then to prove that $\mathcal{MR}_p(\mathcal{X})\geq2\rho'_0(\mathcal{X}).$ If $\mathcal{MR}_p(\mathcal{X})=2$, the
	inequality is obvious, so we can assume $\mathcal{MR}_p(\mathcal{X})<2$ and then the reflexivity of $X.$

	Let $\epsilon\in[0,2]$ such that $\delta_\mathcal{X^*}(\epsilon)=0.$ For such $\epsilon$, according to Hahn-Banach theorem,  there exist two sequences $\{\widetilde{f}_n\}_{n=1}^{\infty}$, $\{\widetilde{g}_n\}_{n=1}^{\infty}\subseteq \mathcal{S}_{\mathcal{X^*}}$ such that $\|\widetilde{f}_n-\widetilde{g}_n\|=\epsilon$ for all $n\geq1$ and $\lim_{n\to{\infty}}\|\widetilde{f}_n+\widetilde{g}_n\|=2.$ 
	
	Consider, for each $n\geq1,~x_n\in \mathcal{S}_\mathcal{X}$ such that $(\widetilde{f}_n+\widetilde{g}_n)(x_n)=$ $\|\widetilde{f}_n+\widetilde{g}_n\|.$ It must be
	$$\lim\limits_{n\to\infty}\widetilde{f}_n(x_n)=\lim\limits_{n\to\infty}\widetilde{g}_n(x_n)=1.$$
	
	Thus, by  Lemma \ref{l2}
	$$\mathcal{MR}_p(\mathcal{X})\geq\lim_{n\to\infty}\widetilde{g}_n(x_n)\|\widetilde{f}_n-\widetilde{g}_n\|=\epsilon.$$
	We have proved that, for any $\epsilon\in[0,2]$ such that $\delta_{\mathcal{X}^*}(\epsilon)=0$, we have $\mathcal{MR}_p(\mathcal{X})\geq\epsilon.$ Therefore,
	$\mathcal{MR}_p(\mathcal{X})\geq \sup \{ \epsilon \in [ 0, 2] : \delta _{\mathcal{X}^{* }}( \epsilon ) = 0\} = 2\epsilon _{0}( \mathcal{X}^{* }) =2\rho' _{\mathcal{X}}( 0) .$
\end{proof}

Since both $\rho'_\mathcal{X}(0)<\frac{1}{\mu(\mathcal{X})}$ and $\rho'_\mathcal{X}(0)<\frac{\mathcal{M}(\mathcal{X})}{2}$ imply normal structure \cite{04}. From Theorem \ref{t3} , the following two results can be obtained: 
\begin{Corollary}
	(i)	Let $\mathcal{X}$ be a Banach space with $\mathcal{MR}_p(\mathcal{X})<\frac{2}{\mu(\mathcal{X})}$, then $\mathcal{X}$ has normal structure.
	
	(ii) Let $\mathcal{X}$ be a Banach space with $\mathcal{MR}_p(\mathcal{X})<\mathcal{M}(\mathcal{X})$, then $\mathcal{X}$ has normal structure.
\end{Corollary}
\subsection*{Acknowledgment}
Many thanks to our \TeX-pert for developing this class file.

\end{document}